\documentclass[12pt,oneside]{amsart}

\usepackage[english]{babel}
\usepackage{cancel}
\usepackage{cite}
\usepackage{latexsym}
\usepackage{amssymb,amsthm,amsmath}
\usepackage{graphicx}
\usepackage{longtable}
\usepackage{xcolor}
\usepackage{enumerate}
\newtheorem{theorem}{Theorem}[section]
\newtheorem{proposition}[theorem]{Proposition}

{\theoremstyle{definition}
\newtheorem*{definition*}{Definition}
\newtheorem{definition}[theorem]{Definition}}
\newtheorem*{proposition*}{Proposition}
\newtheorem*{corollary*}{Corollary}
\newtheorem*{lemma*}{Lemma}
\newtheorem*{remark*}{Remark}

\newcommand{\cX}{\mathcal X}
\newcommand{\cL}{\mathcal L}
\newcommand{\F}{\mathbb {F}}
\newcommand{\fq}{\mathbb {F}_q}
\newcommand{\N}{\mathbb {N}}
\newcommand{\Z}{\mathbb {Z}}
\newcommand{\bP}{\mathbb {P}}
\def\cD{\mathcal D}

\definecolor{dgreen}{rgb}{0.13,0.7,.63}

\def\ariane#1 {\fbox {\footnote {\ }}\ \footnotetext { From Ariane: {\color{red}#1}}}
\def\daniele#1 {\fbox {\footnote {\ }}\ \footnotetext { From Daniele: {\color{blue}#1}}}
\def\lu#1 {\fbox {\footnote {\ }}\ \footnotetext { From Luciane: {\color{dgreen}#1}}}

\title{Pure gaps on curves with many rational places}

\date{}

\begin{document}

\author{Daniele Bartoli, Ariane M. Masuda, Maria Montanucci, and Luciane Quoos}

\address{Dipartimento di Matematica e Informatica, Universit\`a degli Studi di Perugia, Via Vanvitelli 1, Perugia, 06123   Italy}
\email {daniele.bartoli@unipg.it}
\address{Department of Mathematics, New York City College of Technology, CUNY, 300 Jay Street, Brooklyn, NY 11201 USA}

\email{amasuda@citytech.cuny.edu}

\address{Dipartimento di Matematica, Informatica ed Economia, Universit\`a degli Studi della Basilicata, Viale dell'Ateneo
lucano 10,  Potenza, 8500 Italy}

\email{maria.montanucci@unibas.it}

\address{Instituto de Matem\'atica, Universidade Federal do Rio de Janeiro,  Av. Athos da Silveira Ramos 149, Centro de Tecnologia - Bloco C, Ilha do Fund\~ao, Rio de Janeiro, RJ 21941-909  Brazil}

\email{luciane@im.ufrj.br}

\begin{abstract}
We consider the algebraic curve defined by $y^m = f(x)$  where $m \geq 2$  and $f(x)$   is a rational function over $\mathbb{F}_q$.
We extend the concept of pure gap to {\bf c}-gap and obtain a criterion to decide when
an $s$-tuple is a {\bf c}-gap at $s$ rational places on the curve. As an application, we obtain many
families of pure gaps at two rational places on curves with many rational places.
\end{abstract}
\maketitle

\section{Introduction}

Since Goppa introduced the Algebraic Geometric codes in the eighties~\cite{G1982}, a lot of effort has been directed towards obtaining examples of codes with good parameters through different types of algebraic curves. It is well known that codes with optimal parameters are expected from algebraic curves with many rational places over finite fields. These are curves whose numbers of rational places are equal or close to the Hasse-Weil upper bound or other known bounds that may be specific to the curve. For a projective, absolutely irreducible, non-singular algebraic curve $\mathcal{X}$  of genus $g$ over $\fq$, the Hasse-Weil upper bound on the number of $\fq$-rational places is  $q+1+2g\sqrt{q}.$
When this quantity is attained, the curve $\mathcal{X}$ is said to be {\it maximal}. Maximal curves only exist over $\mathbb{F}_{q^2}$.

In  \cite{GKL1993} and \cite{GL1992}  Garcia, Kim and Lax exploit a local property at a rational place on an algebraic curve in order to improve the minimum distance of the code. Specifically, they show that the existence of $\ell$ consecutive gaps at a rational place can be used to  increase the classical upper bound for the minimum distance  by $\ell$ units. In \cite{HK2001} Homma and Kim obtain similar results using gaps at two rational places.
They also define  a special type of gap that they call  {\it pure gap} and obtain further improvements. Then they use pure gaps to refine the parameters of codes constructed from the Hermitian curve over $\mathbb{F}_{q^2}$. In~\cite{CT2005} Carvalho and Torres consider in detail gaps and pure gaps at more than two places. Since then, pure gaps  have been exhaustively studied to show that they provide many other good codes; see~\cite{K1994, YH2017, YH2018, M2001, Mat2004, CMQ2016, BQZ2016}.  The scope of these papers varies depending on the algebraic curve and the number of places considered.

The majority of maximal curves and curves with many rational places has a plane model of Kummer-type.  For those curves with affine equation given by
$$y^m=f(x)^{\lambda} \text{ where }  m \geq 2, \lambda \geq 1 \text{ and } f(x) \text{ is a separable polynomial over } \mathbb{F}_q,$$ general results on gaps and pure gaps can be found in~\cite{ABQ2016, CMQ2016, YH2017, HY2018}.  Applications to codes on particular curves such as the Giulietti-Korchm\'aros curve, the Garcia-G\"{u}neri-Stichtenoth curve, and quotients of the Hermitian curve can be found in~\cite{TC2017,HY2017,YH2017Bis}.
In~\cite{ABQ2016} the authors use a decomposition of certain Riemann-Roch vector spaces due to Maharaj (Theorem~\ref{ThMaharaj}) to describe gaps at one place  arithmetically. The same idea has been used to investigate gaps and pure gaps at several places~\cite{BQZ2016, CMQ2016, HY2018, HY2017, YH2017, YH2018, YH2017Bis}.  Here we continue exploring its capabilities by considering a different setting.

In this work we consider a Kummer-type curve defined by
$$y^m = f(x) \text{ where } m \geq 2 \text{ and } f(x)  \text{ is a rational function over } \mathbb{F}_q.$$
We extend the concept of pure gap to {\it{${\bf c}$-gap}} and provide an arithmetical criterion to decide when an $s$-tuple is a ${\bf c}$-gap at $s$ rational places.
This result is then heavily used to obtain many families of pure gaps at two places on the curves obtained by Giulietti and Korchm\'aros in~\cite{GK2009},  Garcia and Quoos in~\cite{GQ2001}, and Garcia, G\"{u}neri and Stichtenoth in~\cite{GGS2010}. All these curves are known to have many rational places.

In the special case of two rational places, one can also determine the set of pure gaps using a method introduced by  Homma and Kim  in~\cite[Theorem 2.1]{HK2001}.  For $i=1, 2$, let $P_i$  be a rational place and $G(P_i)$ be the set of gaps at $P_i$. There is a bijection $\beta$ from $G(P_1)$ to $G(P_2)$ given by
$$ \beta(n_1) := \min\{  n_2 : (n_1, n_2) \in H(P_1, P_2) \},$$
which can be used to characterize the set of pure gaps at $(P_1,P_2)$ in the following way:
\begin{align}\label{gapssuzuki}
G_0(P_1,P_2)= \{(n_1,n_2) \in G(P_1) \times G(P_2) : n_1 < \beta^{-1}(n_2) \text{ and } \beta(n_1) > n_2\}.
\end{align}
We use this approach to obtain additional families of pure gaps at two places on the Suzuki curve.

We conclude the paper with a summary of the parameters of codes constructed using our families of  pure gaps.

\section{Preliminary results}\label{Sec:Preliminary}

Let $\cX$ be a projective, absolutely irreducible, non-singular algebraic curve of genus $g$ defined over a finite field $\fq $. Let $F=\fq(\cX)$ be its function field over $\fq$. For a function $z$ in $F$, let $(z)$ and $(z)_\infty$ stand for its divisor and  pole divisor, respectively. We denote by $\bP_F$ the set of places of $F$, and by $\cD_F$ the free abelian group generated by the places of $F$. The elements $D$ in $\cD_F$ are called {\it divisors} and can be written as
$$
D=\sum_{P\in \bP_F}n_P\,P\quad \text{ with } n_P\in \Z, \text{ and } n_P=0\text{ for almost all }P\in\bP_F.
$$
The degree of a divisor $D$ is $\deg(D)=\sum\limits_{P\in \bP_F}n_P \deg P$.
The {\emph{Riemann-Roch vector space}} associated to  $D$ is defined by
$$
\cL(D):=\{z\in F : (z)\ge -D\}\cup \{0\}.
$$
We denote by $\ell(D)$ the dimension of $\cL(D)$ as a vector space over the field of constants $\fq$.

Our convention is that    $\N=\{1,2,\dots\}$ and $\N_0=\N\cup\{0\}$.   For a rational place $P$, that is, a place $P$ of degree one, the {\emph{Weierstrass semigroup}} at $P$ is
$$H(P):=\{n\in \N :  (z)_\infty=nP \text{ for some } z\in F\} \cup \{0\}.$$
We say that $n$ is a {\emph{non-gap}} at $P$ if $n\in H(P)$, and a {\emph{gap}} otherwise. Another way to characterize a gap is using the dimension of Riemann-Roch spaces, namely, $n$ is a gap at $P$ if and only if $\ell((n-1)P)=\ell(nP)$. As a consequence, the Weierstrass Gap Theorem asserts that there exist $g$ gaps at $P$ between $1$ and $2g-1$. There is a natural generalization of the notion of gap at distinct rational places $P_1, \dots, P_s$  of $F$ by setting
$$
H(P_1, \dots, P_s):= \{(n_1, \dots , n_s) \in \N_0^s :  (z)_\infty = n_1P_1 + \cdots + n_sP_s \text{ for some } z \in F  \}
$$
as  the Weierstrass semigroup at $(P_1, \dots , P_s)$.  The elements in $$G(P_1, \dots, P_s):=\N_0^s \setminus H(P_1, \dots, P_s)$$ are called  gaps, and there is  a finite number of them.  Gaps at several places can be described in terms of the dimensions of Riemann-Roch spaces:
$(n_1, \dots , n_s)$ is a gap at $(P_1, \dots , P_s)$ if and only if
$$\ell\left(\sum_{i=1}^s n_iP_i- P_j\right)=\ell\left(\sum_{i=1}^s n_iP_i\right) \text{ for some } j \in \{ 1, \dots , s\}.$$

Homma and Kim~\cite{HK2001} introduced the important concept of {\it pure gap}.  An
$s$-tuple $(n_1, \dots, n_s)$ $\in\N_0^s$ is a {\it pure gap} at $(P_1, \dots, P_s)$ if
$$
\ell\left(\sum_{i=1}^s n_iP_i- P_j\right)=\ell\left(\sum_{i=1}^s n_iP_i\right) \text{ for all } j \in \{1, \dots, s\}.
$$
\noindent The set of pure gaps at $(P_1,\ldots,P_s)$ is denoted by $G_0(P_1,\ldots,P_s)$.
Clearly, $G_0(P_1,\ldots,P_s)$ is contained in $G(P_1,\ldots,P_s)$. Carvalho and Torres~\cite[Lemma 2.5]{CT2005} showed that  $(n_1, \dots, n_s)$ is a pure gap at $(P_1, \dots ,P_s)$ if and only if $\ell\left(\sum_{i=1}^s n_iP_i\right)= \ell\left(\sum_{i=1}^s (n_i-1)P_i\right)$.


Now we turn our attention to Kummer extensions. Let $K$ be the algebraic closure of  $\mathbb{F}_q$ and $p$ be the characteristic of $\mathbb F_q$.
\begin{definition}\label{kummer}
A {\emph{Kummer extension}} is an algebraic function field $F/K(x)$ defined by $y^m=f(x)$ where $f(x)$ is a rational function over $K$ with $m\geq 2$ and $p \nmid m.$ For a zero or pole of $f(x)$, let $Q$ be the place of $\mathbb{P}_{K(x)}$ associated to it. To be more precise, we assume the following.
\begin{enumerate}
\item $Q_1,\ldots, Q_r \in \,\mathbb{P}_{K(x)}$ are all places corresponding to zeros and poles of $f(x)$. We denote the
multiplicity of $Q_i$ by  $\lambda_i:= v_{Q_i}(f(x)) \in \mathbb{Z}$.
\item $Q_1,\ldots, Q_s \in \,\mathbb{P}_{K(x)}$ with $s\leq r$ are totally ramified places in the extension $F/K(x).$
\item For $i=1, \dots, r$, let $P_{i, j}$ with $1 \leq j \leq d_i:=\gcd(m, \lambda_i)$  be the  places in $\mathbb{P}_F $ lying over $ Q_i.$ When $Q_i$ is a totally ramified place in $ F/K(x),$ we denote $ P_{i,1}$ simply by $P_i.$
\item The divisor of the function $y$  in $F$  is \\
$$ (y):=\sum_{i=1}^{r} \sum_{j=1}^{d_i} \frac{\lambda_i}{d_i} P_{i,j}=\sum_{i=1}^{s} \lambda_i P_i   + \sum_{i=s+1}^{r} \sum_{j=1}^{d_i} \frac{\lambda_i}{d_i} P_{i,j}.$$
\item For a divisor $D$ of $F$ and any subfield $E\subseteq F$, write
$D = \displaystyle\sum_{R\in \bP_E}\;\sum_{\stackrel{Q\in \bP_F}{Q|R}}\, n_Q\, Q$.
The restriction of $D$ to $E$ is
$$
D_{|_{E}}:= \sum\limits_{R\in \bP_E} \min\,\left\{\left\lfloor\frac{n_Q}{e(Q|R)}\right\rfloor :
{Q|R}\right\}\,R
$$
where $e(Q|R)$ is the ramification index of $Q$ over $R$.
\end{enumerate}
\end{definition}

The following result by Maharaj  is a key ingredient in our arithmetic characterization of ${\bf c}$-gaps at several places.

\begin{theorem}[\!\!{\cite[Theorem 2.2]{M2004}}] \label{ThMaharaj}
Let $F/\fq(x)$ be a Kummer extension of degree $m$ defined by $y^m=f(x)$. Then, for any divisor $D$ of $F$ that is invariant under the action of $Gal(F|\fq(x))$, we have that
\begin{equation*}
 \mathcal{L}(D)= \bigoplus\limits_{t=0}^{m-1} \mathcal L\left(\left[D+(y^t)\right]\Big|_{\fq(x)}\right)\,y^t
 \end{equation*}
where $\left[D+(y^t)\right]\Big|_{\fq(x)}$ denotes the restriction of the divisor $D+(y^t)$ to $\fq(x)$.
\end{theorem}

Pure gaps are related  to the improvement of  the designed minimum distance of Algebraic Geometric codes. For the differential code $C_{\Omega}(D,G)$ associated to the divisors $D$ and $G$, we denote its length, dimension and minimum distance by $n$, $k$ and $d$, respectively;  see~\cite{S2009}. 

We recall that   $\cX$ denotes a projective, absolutely irreducible, non-singular algebraic curve of genus $g$.

\begin{theorem}[\!\!{\cite[Theorem 3.4]{CT2005}}] \label{distmanypoints}
Let $P_1, \dots , P_s,Q_1, \dots, Q_n$ be pairwise distinct $\fq$-rational places on $\cX$, and let $(a_1, \dots , a_s)$, $(b_1, \dots , b_s)\in\N_0^s$ be pure gaps at $(P_1, \dots, P_s)$ with $a_i \leq b_i$ for each $i$. Consider the divisors $D= Q_1+ \cdots + Q_n$ and $G= \sum_{i=1}^s (a_i+b_i-1)P_i$.
If every $(c_1, \dots , c_s)\in\N_0^s$ with $a_i \leq c_i \leq b_i$  for $i\in\{1,\dots,s\}$ is a pure gap at $(P_1, \dots , P_s)$, then $$ d \geq \deg(G) -(2g-2)+\sum_{i=1}^s (b_i-a_i+1).$$
\end{theorem}


\section{${\bf c}$\,-gaps on Kummer extensions}\label{Kummer}

We start this section by generalizing the concept of gap to ${\bf c}$-gap at several places.

\begin{definition}
Let  $P_1, \dots, P_s$ be pairwise distinct $\fq$-rational places on $\mathcal{X}$.
 Given  ${\bf c}=(c_1, \dots, c_s)\in\N_0^s$, we say that ${\bf n}=(n_1, \dots, n_s)\in\N_0^s$ is a {\it ${\bf c}$-gap} at $(P_1, \dots, P_s)$ if
$$\ell \left(\sum_{i=1}^{s} (n_i-c_i)P_i \right) =\ell \left(\sum_{i=1}^{s} n_i P_i\right).$$
\end{definition}

When ${\bf{1}}=(1, \dots, 1)$, a $\bf{1}$-gap at $(P_1, \dots, P_s)$ is a pure gap at $(P_1, \dots, P_s)$. Also, if ${\bf n}$ 
is a $(c_1, \dots, c_s)$-gap at $(P_1, \dots, P_s)$ then   ${\bf n}, {\bf n}-{\bf{1}}, \dots, {\bf n}-\delta\cdot{\bf{1}}$ are all pure gaps at $(P_1, \dots, P_s)$ where $\delta := \min \{ c_i : i=1, \dots, s \}\geq 1$.


Lundell and McCullough proved a bound for the minimum distance $d$ called {\emph{generalized floor bound}}.
Next we state their result using the  ${\bf c}$-gap terminology.


\begin{theorem}[\!\!{\cite[Theorem 3]{LM2006}}]\label{distMany}
Let $P_1, \dots , P_s, Q_1, \dots, Q_n$ be pairwise distinct $\fq$-rational places on $\cX$. Suppose that there exist ${\bf a}=(a_1,\ldots,a_s)$, ${\bf b}=(b_1,\ldots,b_s)$ and ${\bf c}=(c_1,\ldots,c_s) \in \N_0^s$ with each $c_i\leq b_i-1$  such that ${\bf a+c}$ and ${\bf b-1}$ are ${\bf c}$-gaps at $(P_1, \dots, P_s)$. Consider the divisors $D=Q_1+\cdots +Q_n$ and $G=\sum_{i=1}^{s} (a_i+b_i-1)P_i$.
If $\Omega(G-D)\neq \emptyset$ then $$d \geq \deg (G) -(2g-2)+\sum_{i=1}^{s} c_i.$$
\end{theorem}

Theorem \ref{distmanypoints} can be seen as a particular case of Theorem \ref{distMany}. 

There are instances in which we can apply Theorem \ref{distMany}, but not Theorem~\ref{distmanypoints}.  This occurs when some $c_i=0$ and  ${\bf a}$,  ${\bf b}$ are not  pure gaps, or when $a_i \leq b_i$ and $a_j \geq b_j$ for some $i, j$.

\begin{theorem}\label{puregapsmanypoints}
Let $P_1, \dots, P_s \in \mathbb{P}_F$ be pairwise distinct totally ramified places in the Kummer extension $F/K(x)$.
Let ${\bf c}=(c_1, \dots, c_s)\in\N_0^s$.  Then  $(n_1, \dots, n_s) \in\N_0^s$ is a ${\bf c}$-gap at $(P_1, \dots, P_s)$ if and only if for every $t \in \{0,\ldots ,m - 1\}$ exactly one of the two following conditions is satisfied:
\begin{enumerate} [{\normalfont(i)}]
\item $\displaystyle\sum_{i=1}^s \left \lfloor \frac{n_i+t\lambda_i}{m}\right\rfloor  + \sum_{i=s+1}^r \left \lfloor \frac{t\lambda_i}{m}\right\rfloor < 0$
\item $\displaystyle\left \lfloor \frac{n_i+t\lambda_i}{m}\right\rfloor=\left \lfloor \frac{n_i-c_i+t\lambda_i}{m}\right\rfloor$
for all $i \in \{ 1, \dots, s\}.$
\end{enumerate}
\end{theorem}
\begin{proof}
Let $t\in\{0,\ldots,m-1\}$. Using the notation as in Definition \ref{kummer},  we have that
$$\sum_{i=1}^s n_iP_i+(y^t)=\sum_{i=1}^s  \left(n_i+t\lambda_i\right)P_{i}+\sum_{i=s+1}^{r}\sum_{j=1}^{d_i} t\frac{\lambda_i}{d_i} P_{i,j},$$
and the restriction of this divisor to $K(x)$ is
$$
\left[\sum_{i=1}^s n_iP_i+\left(y^t\right) \right]\Bigg|_{{K(x)}} =
\sum_{i=1}^s \left \lfloor \frac{n_i+t\lambda_i}{m}\right\rfloor Q_i + \sum_{i=s+1}^r \left \lfloor \frac{t\lambda_i}{m}\right\rfloor Q_i.
$$
By definition, ${\bf n}$ is a ${\bf c}$-gap at $(P_1, \dots, P_s)$ if
$\ell \left(\sum_{i=1}^{s} (n_i-c_i)P_i \right) =\ell (\sum_{i=1}^{s} n_i P_i).$
By Theorem~\ref{ThMaharaj}, we have that
$$
\mathcal{L}\left(\sum_{i=1}^s n_iP_i\right)=\bigoplus_{t=0}^{m-1}\mathcal{L}\left(\left[\sum_{i=1}^s n_iP_i+\left(y^t\right) \right]\Bigg |_{K(x)} \right)y^t.
$$
So
\begin{align*}
 \ell\left(\sum_{i=1}^s n_iP_i\right) &= \sum_{t=0}^{m-1}\ell\left( \sum_{i=1}^s \left \lfloor \frac{n_i+t\lambda_i}{m}\right\rfloor Q_i + \sum_{i=s+1}^r \left \lfloor \frac{t\lambda_i}{m}\right\rfloor Q_i\right ) \text{ and}\\
\ell\left(\sum_{i=1}^s (n_i-c_i)P_i\right) &=\sum_{t=0}^{m-1} \ell\left( \sum_{i=1}^s \left \lfloor \frac{n_i-c_i+t\lambda_i}{m}\right\rfloor Q_i + \sum_{i=s+1}^r \left \lfloor \frac{t\lambda_i}{m}\right\rfloor Q_i \right ).
\end{align*}
We conclude that ${\bf n}$ is a ${\bf c}$-pure gap at $(P_1,\ldots,P_s)$ if and only if
\begin{align*}
&\ell\left( \sum_{i=1}^s \left \lfloor \frac{n_i+t\lambda_i}{m}\right\rfloor Q_i + \sum_{i=s+1}^r \left \lfloor \frac{t\lambda_i}{m}\right\rfloor Q_i \right)\\
&-\ell\left( \sum_{i=1}^s \left \lfloor \frac{n_i-c_i+t\lambda_i}{m}\right\rfloor Q_i + \sum_{i=s+1}^r \left \lfloor \frac{t\lambda_i}{m}\right\rfloor Q_i\right)=0
\end{align*}
for all $t\in \{0,\ldots,m-1\}$. Since $\fq(x)$ has genus $0$, the latter identity holds if and only if for all $t\in \{0,\ldots,m-1\}$ either
$$\sum_{i=1}^s \left \lfloor \frac{n_i+t\lambda_i}{m}\right\rfloor  + \sum_{i=s+1}^r \left \lfloor \frac{t\lambda_i}{m}\right\rfloor  < 0$$
or
$$ \sum_{i=1}^s \left \lfloor \frac{n_i+t\lambda_i}{m}\right\rfloor + \sum_{i=s+1}^r \left \lfloor \frac{t\lambda_i}{m}\right\rfloor  \geq 0 \quad\text{ and }\quad \left \lfloor \frac{n_i+t\lambda_i}{m}\right\rfloor=\left \lfloor \frac{n_i-c_i+t\lambda_i}{m}\right\rfloor$$

\noindent for $i=1, \dots, s$. This concludes the proof.
\end{proof}

Our next goal is  to obtain families of pure gaps on specific curves using Theorem~\ref{puregapsmanypoints}.


\subsection{Pure gaps on the GK curve}
The Giulietti-Korchm\'aros curve  over $\mathbb{F}_{q^6}$ is a non-singular curve in ${\rm PG}(3,K)$ defined by the affine equations:
\[GK:
\begin{cases}
Y^{q+1} = X^q+X,\\
Z^{q^2-q+1} =Y^{q^2}-Y.
\end{cases}
\]
It has genus $g=\frac{(q^3+1)(q^2-2)}{2}+1$, and the number  of its $\mathbb{F}_{q^6}$-rational places is $q^8-q^6+q^5+1$. The GK curve first appeared in~\cite{GK2009} as a maximal curve over $\mathbb{F}_{q^6}$, since the latter number coincides with the Hasse-Weil bound, $q^6+2gq^3+1$.
The GK curve is the first example of a maximal curve that is not $\mathbb{F}_{q^6}$-covered by the Hermitian curve, provided that $q>2$.

A plane model for the GK curve is given by the affine equation:
\begin{equation*}
y^{q^3+1}= (x^q+x)((x^q+x)^{q-1}-1)^{q+1}.
\end{equation*}
Let $F=\mathbb{F}_{q^6}(x,y)/\mathbb{F}_{q^6}$ be its function field.  We denote by $P_{\infty}$ the only place of $F$ associated to the pole of $x$.

\begin{proposition}\label{PropGK1} Let $P_1$ and $P_2$ be two totally ramified places different from $P_{\infty}$ on the GK curve.
For $\alpha\in\{0,\dots, q^2-3\}$ and $\beta\in\{0,1\}$,  let
\begin{align*}
n_1 &:= 1+\alpha(q^3+1)+\beta (q^2-q+1) \text{ and } \\
n_2 &:= q^5-3q^3+2q^2-q-1-\alpha(q^3+1)-\beta (q^2-q+1).
\end{align*}
Then the pair $(n_1,n_2)$ is a pure gap at $(P_1, P_2)$.
\end{proposition}
\begin{proof}The places $P_1$ and $P_2$ are zeros of $x^q+x$, so $\lambda_1=\lambda_2=1$.
For $i\in\{1,2\}$ and $t \in \{0,\dots, q^3\}$,  we have that $\left \lfloor \frac{n_i+t}{q^3+1}\right \rfloor \neq \left \lfloor \frac{n_i+t-1}{q^3+1}\right \rfloor$ if
and only if $n_i +t \equiv 0\pmod{q^3+1}$.  Furthermore,
$$
n_1+t \equiv 0\pmod{q^3+1}
\quad\Longleftrightarrow\quad
\begin{cases}
t=q^3 & \text{ if } \beta=0 \\
t=(q^2+1)(q-1) & \text{ if } \beta=1
\end {cases}$$
\noindent and
$$
n_2+t \equiv 0 \pmod {q^3+1}
\quad\Longleftrightarrow \quad
\begin{cases}
t=(q^2+1)(q-1) & \text{ if }  \beta=0 \\
t=q^3 & \text{ if }  \beta=1.\\
\end{cases}
$$
It remains to show that the condition Theorem~\ref{puregapsmanypoints}(i) holds in the above cases.  We compute
\begin{align*}
& \left \lfloor \frac{n_1+t}{q^3+1}\right \rfloor +\left \lfloor \frac{n_2+t}{q^3+1}\right \rfloor
+(q^2-q) \left \lfloor \frac{t}{q^2-q+1}\right \rfloor
+\left \lfloor \frac{-t q^3}{q^3+1}\right \rfloor \\
=\quad & \begin{cases}
q^2-1+(q^2-q)q-q^3=-1 &\text{ if } t=q^3\\
q^2-2+(q^2-q)(q-1)-q^3+q^2-q+1=-1 & \text{ if } t=(q^2+1)(q-1).
\end{cases}
\end{align*}
Therefore $(n_1, n_2)$ is a pure gap at $(P_1,P_2)$.
\end{proof}

\begin{proposition}
Let $P_1$ and $P_2$ be two totally ramified places different from $P_{\infty}$ on the GK curve.
The pair $(1,2g-2)$ is a $(1,0)$-gap, but not a pure gap, at $(P_1,P_2)$.
\end{proposition}

\proof
We start by investigating when  Theorem \ref{puregapsmanypoints}(ii) does not hold. For $i=1$, we have that
$$\left \lfloor \frac{1+t}{q^3+1}\right \rfloor \neq\left \lfloor \frac{t}{q^3+1}\right \rfloor$$
if and only if $t=q^3$. In this case (i) becomes
\begin{align*}
& \left \lfloor \frac{1+q^3}{q^3+1}\right \rfloor +\left \lfloor \frac{q^5-q^3+q^2-2}{q^3+1}\right \rfloor
+(q^2-q) \left \lfloor \frac{q^3}{q^2-q+1}\right \rfloor
+\left \lfloor \frac{-q^6}{q^3+1}\right \rfloor \\
= \quad&1+q^2-2+(q^2-q)q-q^3=-1<0.
\end{align*}

On the other hand, $(1,2g-2)$ is a not a $(1,1)$-gap since for $t=0$ and $i=2$ we have
$$\left \lfloor \frac{1}{q^3+1}\right \rfloor +\left \lfloor \frac{q^5-2q^3+q^2-2}{q^3+1}\right \rfloor
>0$$
and
$$\left \lfloor \frac{q^5-2q^3+q^2-2}{q^3+1}\right \rfloor \neq \left \lfloor \frac{q^5-2q^3+q^2-3}{q^3+1}\right \rfloor.$$
\endproof

\begin{proposition}
Let $P_1$ and $P_2$ be two totally ramified places different from $P_{\infty}$ on the GK curve. For $i\in\{1,2\}$, let $$n_i:=\alpha_i(q^3+1)+\beta_i(q^2-q+1)+\gamma_i$$
where  $\alpha_i\in\{0,\ldots ,q^2-2\}$, $\beta_i\in\{0,\ldots,q\}$ and $\gamma_i\in\{1,\ldots,q^2-q\}$. Let
$$\delta_{i,j}:=
\begin{cases}
1 &  \text{ if  } \beta_i(q^2-q+1)+\gamma_i \leq \beta_j(q^2-q+1)+\gamma_j\\
 0 & \text{ otherwise.}
 \end{cases}
 $$
Then $(n_1,n_2)$ is a pure gap at $(P_1,P_2)$ if and only if $$q^2+1-\delta_{1,2}-\delta_{2,1}-\alpha_1-\alpha_2>\max\{\beta_1+\gamma_1,\beta_2+\gamma_2\}.$$
\end{proposition}
\proof
By Theorem \ref{puregapsmanypoints}, $(n_1,n_2)$ is a pure gap at $(P_1,P_2)$ if and only if, whenever $\left \lfloor \frac{n_i+t_i}{q^3+1} \right \rfloor \ne \left \lfloor \frac{n_i+t_i-1}{q^3+1} \right \rfloor$ for some $i\in\{1,2\}$ and $t_i \in \{0,\ldots,q^3\}$, one has
\begin{equation} \label{conto}
\left \lfloor \frac{n_1+t_i}{q^3+1} \right \rfloor+\left \lfloor \frac{n_2+t_i}{q^3+1} \right \rfloor+(q^2-q) \left \lfloor \frac{t_i}{q^2-q+1}\right \rfloor + \left \lfloor \frac{-q^3t_i}{q^3+1}\right \rfloor <0.
\end{equation}
Clearly, $\left \lfloor \frac{n_i+t_i}{q^3+1} \right \rfloor \ne \left \lfloor \frac{n_i+t_i-1}{q^3+1} \right \rfloor$ if and only if $n_i+t_i \equiv 0 \pmod {q^3+1}$. In particular, since $\gamma_i \geq 1$, we get $t_i=(q-\beta_i)(q^2-q+1)+q^2-q+1-\gamma_i$ for $i\in\{1,2\}$. First, we note that $t_i\geq t_j$ if and only if $\delta_{i,j}=1$. For $j \ne i$, we obtain that

$$\left \lfloor \frac{{n_j}+{t_i}}{q^3+1} \right \rfloor =\alpha_j+\left \lfloor \frac{\beta_j(q^2-q+1)+\gamma_j - \beta_i(q^2-q+1)-\gamma_i + q^3+1}{q^3+1} \right \rfloor= \alpha_j+\delta_{i,j}.$$

\noindent Thus, by \eqref{conto}, we have
$$\alpha_1+\delta_{i,1}+\alpha_2+\delta_{i,2}+(q^2-q)(q-\beta_i)-t_i<0,$$
which is equivalent to
$$\beta_i+\gamma_i<q^2+1-\delta_{1,2}-\delta_{2,1} -\alpha_1-\alpha_2.$$
Hence, we obtain that $(n_1,n_2)$ is a pure gap at $(P_1,P_2)$ if and only if $q^2+1-\delta_{1,2}-\delta_{2,1}-\alpha_1-\alpha_2>\beta_i+\gamma_i$ for  $i\in\{1,2\}$.
\endproof

\subsection{Pure gaps on curves with many rational places}

We now consider curves with many rational places that appeared in \cite{GQ2001}.
The first curve is defined over $\mathbb{F}_{q^{2n}}$ by
$$
\mathcal{X}_1 : y^m= (x^{q^n}-x)^{q^{n}-1}$$
 where  $m \mid (q^{2n}-1)$  and  $\gcd(m, q^n-1)=1$.
This curve has genus $g=(q^{n}-1)(m-1)/2$ and $N=(q^{2n}-q^{n})m+(q^{n}+1)$ rational places
over $\mathbb{F}_{q^{2n}}$.

\begin{proposition}\label{PropX_1}
Suppose $q$ is even. On the curve $\mathcal{X}_1$ defined by
$$y^{q^n+1}= (x^{q^n}-x)^{q^{n}-1},$$
let  $P_\infty$ be the unique place at infinity and $P_1$  be a totally ramified place different from $P_{\infty}$.
We have that:
\begin{enumerate}[{\normalfont(i)}]
\item $\left(1,(q^n-2)(q^n+1)\right)$ is a $(1,0)$-gap, but not a pure gap, at $(P_\infty, P_1)$, and
\item $\left(1,(q^n-3)(q^n+1)+\alpha\right)$  is a pure gap at $(P_\infty, P_1)$ for $\alpha\in\{1,2\}$.
\end{enumerate}
\end{proposition}
\proof
 Let $(n_0, n_1):=\left(1,(q^n-2)(q^n+1)\right)$ and $(c_0, c_1):=(1,0)$.
We have that $\lambda_0=-q^n(q^n-1)$ and $\lambda_i=q^n-1$ for $i\in\{1, \dots q^n\}$. For $i=0$ and $t \in \{ 0, \dots, q^n\}$, we have
$$\left\lfloor \frac{n_0+t\lambda_0}{q^n+1}\right\rfloor \neq \left\lfloor \frac{n_0-1+t\lambda_0}{q^n+1}\right\rfloor$$
if and only if $n_0+t\lambda_0 \equiv 1-q^n t(q^n-1)\equiv 0 \pmod {q^n+1}$, that is, $t=(q^n+2)/2$.  In this case,
$$
\left\lfloor \frac{n_0+t\lambda_0}{q^n+1}\right\rfloor+ \left\lfloor \frac{n_1+t\lambda_1}{q^n+1} \right \rfloor +(q^n-1) \left\lfloor  \frac{t\lambda_1}{q^n+1}\right \rfloor
=-\frac{q^{2n}}{2}+\frac{3(q^n-2)}{2}+ (q^n-1)\frac{q^n-2}{2}=-2<0.
$$
On the other hand, $(1,(q^n-2)(q^n+1))$ is not a $(0,1)$-gap at $(P_\infty, P_1)$ since for $i=1$ and $t=0$ we have that
$$q^n-2=\left\lfloor \frac{n_1+t\lambda_1}{q^n+1}\right\rfloor\neq \left\lfloor \frac{n_1+t\lambda_1-1}{q^n+1}\right\rfloor=q^n-3$$
and
$$\left\lfloor \frac{n_0+t\lambda_0}{q^n+1}\right\rfloor+ \left\lfloor \frac{n_1+t\lambda_1}{q^n+1} \right \rfloor +(q^n-1) \left\lfloor  \frac{t\lambda_i}{q^n+1}\right \rfloor=q^n-2\geq 0.$$
We conclude that $(n_0,n_1)$ is not a pure gap.

Now suppose that $(n_0, n_1):=(1,(q^n-3)(q^n+1)+\alpha)$ and $(c_0, c_1):=(1,1)$. As in the previous case, we have
$$\left\lfloor \frac{n_0+t\lambda_0}{q^n+1}\right\rfloor \neq \left\lfloor \frac{n_0-1+t\lambda_0}{q^n+1}\right\rfloor$$
if and only if $t=\frac{q^{n}+2}{2}$. Also,
$$\left\lfloor \frac{n_1+t\lambda_i}{q^n+1}\right\rfloor\neq \left\lfloor \frac{n_1+t\lambda_i-1}{q^n+1}\right\rfloor$$
if and only if $n_1+t\lambda_i \equiv (q^n-3)(q^n+1)+\alpha+(q^n-1) t\equiv 0 \pmod {q^n+1}$, that is, $t=1$ for $\alpha=2$, and  $t=(q^{n}+2)/2$ for $\alpha=1$.
In these cases
\begin{eqnarray*}
&&\left\lfloor \frac{n_0+t\lambda_0}{q^n+1}\right\rfloor+ \left\lfloor \frac{n_1+t\lambda_1}{q^n+1} \right \rfloor +(q^n-1) \left\lfloor  \frac{t\lambda_1}{q^n+1}\right \rfloor \\
&=&
\begin{cases}
-\frac{q^{2n}}{2}+ (q^n-3)+ \frac{q^n-2}{2}+(q^n-1)\frac{q^n-2}{2}=-3 & \text{ if }t=(q^{n}+2)/2\\
-q^n+1+q^n-3+\left \lfloor\frac{\alpha}{2}\right\rfloor\leq -1  & \text{ if } t=1\\
\end{cases}
\end{eqnarray*}
So $(n_0,n_1)$ is a pure gap.
\endproof

\noindent
The second curve we consider in this section is defined over $\mathbb F_{q^2}$ by
$$
\mathcal{X}_2: \,
y^m=\frac{(x^{q+1}+x+1)^q}{x^{q+1}+x^q+1}
$$
 where $m$ a divisor of  $q^2-1$.
The common roots of the numerator and the denominator belong to $\F_q$ and satisfy $x^2+x+1=0$. We denote by $d$ the $\gcd(m,q-1)$.

\medskip

\noindent
{\sc Case 1:} If $q\equiv0 \pmod 3$ then
$$
g=(q-1)(m-1)+\frac{m-d}2\quad\text{ and }\quad N=q^2m+d.
$$

\noindent
{\sc Case 2:} If $q\equiv 1\pmod 3$ then
\begin{align*}
g &= (q-2)(m-1)+(m-d) \quad\text{ and} \\
N &= \begin{cases}
(q^2-1)m+2d\quad & \text{ if } (q-1)/d\equiv 0 \pmod{3}\\
(q^2-1)m \quad  &\text{ if } (q-1)/d\not\equiv0 \pmod{3}.
\end{cases}
\end{align*}

\noindent
{\sc Case 3:} If $q\equiv 2\pmod 3$ then
$$
g=q(m-1)\quad \text{ and } \quad N=(q^2+1)m.
$$
\begin{proposition}\label{PropX_2}
Let $q>3$ and $m=q^2-1$.   Let $P_1$ and $P_2$ be two totally ramified places corresponding to the roots of $x^{q+1}+x+1$
on the curve $\mathcal{X}_2$ that are different from $P_{\infty}$. For
$\alpha\in\{0,\ldots,2q-5\}$ and  $\beta\in\{0,1\}$, the pair
$$(n_1, n_2):=\left(q-1+\alpha(q^2-1)+\beta q,2q^3-5q^2+4-\alpha(q^2-1)-\beta q\right)$$  is a  pure gap at $(P_1,P_2)$.
\end{proposition}
\proof
In this case $\lambda_i=q$ for $q-1$ places associated to the roots of  $x^{q+1}+x+1$ that are not roots of $x^2+x+1$, $\lambda_i=q-1$ for two places associated to the roots of $x^2+x+1$, $\lambda_i=-1$ for $q-1$ places associated to the roots of  $x^{q+1}+x^q+1$ that are not roots of $x^2+x+1$, and $\lambda_i=-(q^2-1)$ for the last place (the pole $P_{\infty}$). We have that
$$\left \lfloor \frac{n_1+tq}{q^2-1}\right \rfloor \neq \left \lfloor \frac{n_1+tq-1}{q^2-1}\right \rfloor \iff t= \begin{cases} q-1 & \text{ if }  \beta=0\\ q-2 & \text{ if }  \beta=1 \end{cases}$$
and
$$\left \lfloor \frac{n_2+tq}{q^2-1}\right \rfloor \neq \left \lfloor \frac{n_1+tq-1}{q^2-1}\right \rfloor \iff t= \begin{cases}q-2 & \text{ if } \beta=0\\ q-1&  \text{ if }\beta=1 \end{cases}$$

Let
\begin{eqnarray*}
\varphi(t)&=&\left \lfloor \frac{n_1+tq}{q^2-1}\right \rfloor+\left \lfloor \frac{n_2+tq}{q^2-1}\right \rfloor+(q-3)\left \lfloor \frac{tq}{q^2-1}\right \rfloor
+2\left \lfloor \frac{t(q-1)}{q^2-1}\right \rfloor\\
&&+(q-1)\left \lfloor \frac{-t}{q^2-1}\right \rfloor+\left \lfloor \frac{-t(q^2-1)}{q^2-1}\right \rfloor.
\end{eqnarray*}
Suppose $\beta=0$. If $t=q-1$ then
$$ \varphi(t)=\alpha+1+2q-4-\alpha-(q-1)-(q-1)=-1<0.$$
If $t=q-2$ then
$$\varphi(t)=\alpha+2q-4-\alpha-(q-1)-(q-2)=-1<0.$$
Suppose now $\beta=1$. If $t=q-1$ then
$$\varphi(t)=\alpha+1+2q-4-\alpha-(q-1)-(q-1)=-1<0.$$
If $t=q-2$ then
$$\varphi(t)=\alpha+1+2q-5-\alpha-(q-1)-(q-2)=-1<0.$$
\endproof

%
%
%
%
%
%

\subsection{Pure gaps on the GGS curve}

Let $n$ be an odd integer. The Garcia-G\"uneri-Stichtenoth curve was introduced in \cite{GGS2010}.  It is defined by the equations:
\begin{equation*}
GGS: \begin{cases}
X^q + X = Y^{q+1}\\
Y^{q^2}-Y= Z^m
\end{cases}
\end{equation*}
where $m= (q^n+1)/(q+1)$.
This curve has genus $g=\frac{1}{2}(q-1)(q^{n+1}+q^n-q^2)$, and it is $\mathbb F_{q^{2n}}$-maximal.
A plane model for the GGS curve can be given by
$$
Z^{q^n+1}=\frac{(X^{q^2}-X)^{q+1}}{(X^q+X)^q}=(X^q+X) \left( \prod_{\stackrel{a \in \mathbb F_{q^{2}}}{ \mathrm{Tr}(a) \ne 0}} (X- a) \right)^{q+1}
$$
where $\mathrm{Tr}$ is the trace function from  $\mathbb{F}_{q^{2n}}$ to $\mathbb{F}_{q^{2}}$.

\begin{proposition} \label{pgGGS}
On the GGS curve, let $P_1$ and $P_2$ be two totally ramified places  that are different from $P_{\infty}$. Let $\alpha\in\{0,\ldots,q^2-3\}$ and $\beta\in\{0,1\}$. For $n \geq 5$,  if
\begin{align*}
n_1:=&(\beta+1)q^{n-3}(q^2-q+1)+\alpha(q^n+1) \text{ and }\\
n_2:=&(q^2-3)(q^n+1)+3q^{n-3}(q^2-q+1)-(\beta+1)q^{n-3}(q^2-q+1)-\alpha(q^n+1),
\end{align*}
then the pair  $(n_1,n_2)$ is a  pure gap at $(P_1,P_2)$.
\end{proposition}

\begin{proof}
The places $P_1$ and $P_2$ are zeros of $x^q + x$, and so $\lambda_1 = \lambda_2 = 1$. Let $t \in \{0,\ldots,q^n\}$. We have that
$$\left \lfloor \frac{n_1+t}{q^n+1} \right \rfloor \ne \left \lfloor \frac{n_1+t-1}{q^n+1} \right \rfloor$$
if and only if $n_1+t \equiv 0 \pmod {q^n+1}$, which is equivalent to
$$\begin{cases}  t=q^n+1-q^{n-1}+q^{n-2}-q^{n-3} & \text{ if  } \beta=0  \\  t=q^n+1-2q^{n-1}+2q^{n-2}-2q^{n-3} & \text{ if  } \beta=1.  \end{cases}$$
Analogously,
$$\left \lfloor \frac{n_2+t}{q^n+1} \right \rfloor \ne \left \lfloor \frac{n_2+t-1}{q^n+1} \right \rfloor$$
if and only if $n_2+t \equiv 0 \pmod {q^n+1}$, which is equivalent to
$$\begin{cases}    t=q^n+1-2q^{n-1}+2q^{n-2}-2q^{n-3} &\text{ if } \beta=0 \\ t=q^n+1-q^{n-1}+q^{n-2}-q^{n-3} &\text{ if } \beta=1. \end{cases}$$
We need to verify  that Theorem~\ref{puregapsmanypoints}(i) holds for these values of $t$.  Indeed, we have that
$$\left \lfloor \frac{n_1+t}{q^n+1} \right \rfloor+\left \lfloor \frac{n_2+t}{q^n+1} \right \rfloor+(q^2-q)\left \lfloor \frac{t}{(q^n+1)/(q+1)} \right \rfloor+\left \lfloor \frac{-tq^3}{q^n+1} \right \rfloor=$$
$$ \begin{cases} q^2-1+(q^2-q)(q-1)-q^3+q^2-q=-1 & \text{ if  } t=q^n+1-q^{n-1}+q^{n-2}-q^{n-3} \\ q^2-2+(q^2-q)(q-2)-q^3+2q^2-2q+1=-1 & \text{ if  }  t=q^n+1-2q^{n-1}+2q^{n-2}-2q^{n-3}.\end{cases}$$
\end{proof}

\begin{proposition}\label{PropGGS1}
On the GGS curve,  let $P_\infty$ be the unique place at infinity and $P_1$  be a totally ramified place different from $P_{\infty}$.
For $\alpha\in\{0,\ldots,q^2-2\}$, the pair
$$(n_1,n_2):= \left(1+\alpha(q^n+1),1+(q^2-2)(q^n+1)+q^n-2q^3+1-(1+\alpha(q^n+1))\right)$$
is a pure gap at $(P_\infty,P_1)$.
\end{proposition}

\begin{proof}
The place $P_1$ is a zero of $x^q + x$, so $\lambda_1 = 1$, while $\lambda_{\infty}=-q^3$. Let $t \in \{0,\ldots,q^n\}$.
We have that
$$\left \lfloor \frac{n_1+t}{q^n+1} \right \rfloor \ne \left \lfloor \frac{n_1+t-1}{q^n+1} \right \rfloor$$
if and only if $t=q^n$.
Analogously,
$$\left \lfloor \frac{n_2-q^3t}{q^n+1} \right \rfloor \ne \left \lfloor \frac{n_2-q^3t-1}{q^n+1} \right \rfloor$$
if and only if $2q^3+q^3t \equiv 0 \pmod {q^n+1}$, which is equivalent to $t=q^n-1$.
In both cases, we have that
$$\left \lfloor \frac{n_1+t}{q^n+1} \right \rfloor+\left \lfloor \frac{n_2-q^3t}{q^n+1} \right\rfloor+(q^2-q)\left \lfloor \frac{t}{(q^n+1)/(q+1)} \right \rfloor=-1<0.$$
The claim now follows from Theorem \ref{puregapsmanypoints}.
\end{proof}

We note that the structure of pure gaps at $(P_\infty,P_1)$ can be different at $(P_1,P_2)$. In fact,  the automorphism group of $GGS$ acts on the $q$ places, $P_1,\ldots,P_q$, while $P_\infty$ is a fixed place; see \cite{GOS,GMP}.
\section{Pure gaps on the Suzuki curve}\label{Suzuki}

Let $\mathcal{S}$ denote a non-singular model of the projective curve over $\mathbb{F}_q$ defined by
$$\mathcal{S}: y^q-y=x^{q_0}(x^q-x)$$
 where $q_0=2^n$ and $q=2q_0^2=2^{2n+1}$ for some positive integer $n$. The genus of $\mathcal{S}$ is  $q_0(q-1)$. We note that $\mathcal{S}$ has ($q^2+1$) $\mathbb{F}_q$-rational places, namely, there are  $q^2$ places centered at points of type $P_{a,b}:=(a,b,1)$ where $a,b \in \mathbb{F}_q$ in addition to a single places at infinity, $P_\infty$. The curve $\mathcal{S}$ is known as the Suzuki curve because the full automorphism group of $\mathcal{S}$ is isomorphic to the Suzuki group $Sz(q)$ of order $q^2(q^2+1)(q-1)$. Moreover, $Aut(\mathcal{S})$ acts on $\mathcal{S}(\mathbb{F}_q)$ as $Sz(q)$ acts on the points of the Suzuki-Tits ovoid in $\mathrm{PG}(3,q)$; see \cite[Theorem 12.13]{HKT}.

Let $P_1$ and $P_2$ be distinct $\mathbb{F}_q$-rational places on $\mathcal{S}$. We assume that $P_1=P_\infty$ and $P_2=P_{0,0}$. This assumption is not a restriction as $Aut(\mathcal{S}) \cong Sz(q)$ acts $2$-transitively on  $\mathcal{S}(\mathbb{F}_q)$. It turns out that
$$G(P_1)=G(P_2)=\N_0 \setminus \left \langle q, q+q_0, q+2q_0, q+2q_0+1\right \rangle;$$
see \cite{HS}. Following the notation used by Matthews in \cite{Mat2004S}, given $n_i \in G(P_i)$ with $i\in\{1,2\}$, we write
\begin{equation}\label{alpha}
n_i= r_i \left(q+2q_0+1\right) + m_iq_0+s_i \quad\text{ where }\quad r_i = \left \lfloor \frac{n_i}{q+2q_0+1} \right \rfloor
\end{equation}
and $m_i$, $s_i$ are non-negative integers such that $0 \leq s_i \leq q_0-1$. Set
\begin{equation*} 
j_i:= \begin{cases} s_i+q_0  & \text{ if } 0 \leq s_i \leq \left\lfloor (m_i-1)/2 \right\rfloor +1 \\ s_i  & \text{ if } \left\lfloor (m_i-1)/2 \right\rfloor +2 \leq s_i \leq q_0-1.\end{cases}
\end{equation*}
For $n \in G(P_1)$, let
$$\beta(n):= \min \{m \in G(P_2) : (n, m) \in H(P_1,P_2)\}.$$
The function $\beta$ is a bijection between $G(P_1)$ and $G(P_2)$. Matthews showed that  $\beta({n_i})=2g-1+q-(q-1)j_i-n_i$ for the Suzuki curve. 
We recall Equation~\eqref{gapssuzuki}:
\begin{equation*}
G_0(P_1,P_2)= \{(n_1,n_2) \in G(P_1) \times G(P_2) : n_1 < \beta^{-1}(n_2) \text{ and } \beta(n_1) > n_2 \}.
\end{equation*}
Since $Aut(\mathcal{S}) \cong Sz(q)$ acts $2$-transitively on  $\mathcal{S}(\mathbb{F}_q)$, we have that $G_0(P_1, P_2)=G_0(P_2, P_1)$. So $(n_1,n_2) \in G_0(P_1,P_2)$ if and only if $(n_2,n_1) \in G_0(P_1,P_2)$. Therefore, by~\cite[Lemma 2.5]{CT2005}, it turns out that
\begin{eqnarray} \label{puregap}
G_0(P_1,P_2) = \{(n_1,n_2) \in G(P_1) \times G(P_2) : n_2 < \beta(n_1)  \text{ and }n_1 < \beta(n_2)\}.
\end{eqnarray}




\begin{proposition}\label{pg}
Let $P_1,P_2 \in \mathcal{S}(\mathbb{F}_q)$ with $P_1\neq P_2$. If
\begin{equation*} \label{a1a2}
n_1 := \epsilon(q+2q_0+1)+1 \quad\text{ and }\quad n_2:= 2g-q-1-\epsilon(q+2q_0+1)
\end{equation*}
then $(n_1,n_2)$ and $(n_2,n_1)$ are pure gaps at $(P_1,P_2)$ for every $\epsilon\in\{0,\dots, 2q_0-3\}$.
\end{proposition}
\begin{proof}
First let us show that $n_1$ and $n_2$ belong to $$G(P_1)=G(P_2)=\N_0 \setminus \left \langle q, q+q_0, q+ 2q_0, q+2q_0+1\right \rangle.$$
Suppose that $n_1 = Aq+B(q+q_0)+C(q+ 2q_0)+D(q+2q_0+1)$ for $A,B,C,D\in\mathbb N_0$. Then $n_1\equiv D\equiv \epsilon+1\pmod{q_0}$.  If $\epsilon+1<q_0$ then
$$\epsilon(q+2q_0+1)+1 =Aq+B(q+q_0)+C(q+ 2q_0)+(\epsilon+1)(q+2q_0+1)$$
would imply that $0=(A+B+C+1)q+(B+2C+2)q_0\geq q+2q_0$, which is not possible.
Therefore $\epsilon+1 \geq q_0$,
and so
$$n_1=Aq+B(q+q_0)+C(q+2q_0)+(\epsilon+1-q_0)(q+2q_0+1),$$
with $0\leq A,B,C\leq q_0-1$.
This yields
$$(q_0-1)(q+2q_0+1)+1=Aq+B(q+q_0)+C(q+2q_0),$$
that is,
$$2q_0^3-q_0=2Aq_0^2+B(2q_0^2+q_0)+C(2q_0^2+2q_0),$$
or equivalently,
$$2q_0^2-1=2Aq_0+B(2q_0+1)+C(2q_0+2).$$ Therefore
$2q_0^2-1\equiv B+2C\pmod{2q_0}$. The only possibility is $B=2q_0-2C-1$, which gives
\begin{eqnarray*}
2q_0^2-1=2Aq_0+4q_0^2-1-2Cq_0 &\iff& 2q_0^2+2A q_0=2Cq_0\\
&\iff& q_0+A=C,
\end{eqnarray*}
a contradiction, since $C\leq q_0-1$ and $A\geq 0$. Finally, we note that  $n_2$ can be written as
$$n_2=\epsilon^{\prime}(q+2q_0+1)+1,$$
with $\epsilon^{\prime}=2q_0-3-\epsilon$. The same argument as above shows that $n_2\in G(P_2)$.

Now we need to show that $(n_1,n_2)$ is a pure gap at $(P_1,P_2)$. By the definition of $n_1$, we have that
$m_1q_0+s_1=1$ and hence
$(m_1,s_1) = (0,1)$ and $j_1 = 1$.
 On the other hand, we can write $n_2$ as
$$n_2=(2q_0-\epsilon-3)(q+2q_0+1)+2q_0+2,$$
and so $m_2q_0+s_2=2q_0+ 2$. Thus,
$$(m_2,s_2) = \begin{cases} (3,0) &\text{ if }  q_0=2\\ (2,2) & \text{ if }  q_0>2\\ (0,1)  \end{cases} \quad\textup{ and }\quad  j_2 =  2.$$
\noindent For $k,\ell \in \{1, 2\}$ with $k \neq \ell$,  we obtain that
$a_k < \gamma(a_\ell)$ if and only if
$$
2g-q-a_\ell< \\2g-1+q-(q-1)j_{\ell}-a_\ell
$$  if and only if $-q<(q-1)(1-j_{\ell})$, which is always satisfied, given that $j_{\ell}\in\{1,2\}$.
\end{proof}

The pure gaps $(n_1,n_2)$  in the previous result satisfy $n_1+n_2=2g-q$.  Next we will show that this
amount is actually the largest possible value for the sum of the components of any pair of pure gaps on the Suzuki curve.
This can be easily checked when $q_0=2$.


\begin{proposition} \label{degmax}
Let $q_0>2$ and $P_1,P_2 \in \mathcal{S}(\mathbb{F}_q)$ with $P_1\neq P_2$. Then
$$\max \{n_1+n_2 : (n_1,n_2) \in G_0(P_1,P_2)\}=2g-q.$$
\end{proposition}

\begin{proof}
Suppose by contradiction that $(n_1,n_2) \in G_0(P_1,P_2)$ is such that $n_1+n_2>2g-q$. By \eqref{puregap}, we get that
$$\begin{cases}  n_1 < 2g-1+q-(q-1)j_2-n_2 \\ n_2 < 2g-1+q-(q-1)j_1-n_1. \end{cases}$$
So
\begin{eqnarray}\label{desig}
&&2g-q+1 \leq n_1+n_2 <2g-1+q-(q-1)\left(\dfrac{j_1+j_2}{2}\right) \\
&\Longrightarrow & (q-1)\left(\dfrac{j_1+j_2}{2}\right) < 2q-2 \nonumber\\
&\Longrightarrow & j_1+j_2 < 4. \nonumber
\end{eqnarray}
Therefore, either we have $j_1=j_2=1$ or we may suppose that $j_1=2$ and $j_2=1$.
In the former case, we have  $s_1=s_2=1$ and $m_1=m_2=0$. Then, by using (\ref{alpha}), the inequalities in (\ref{desig}) become
\begin{eqnarray*}
&& 2q_0(q-1)-q+1 \leq (r_1+r_2)\left(q+2q_0+1\right)+2 < 2q_0(q-1) \\
&\Longrightarrow &  4q_0^3-2q_0^2-2q_0-1 \leq  (r_1+r_2)(2q_0^2+2q_0+1) < 4q_0^3-2q_0-2 \\
&\Longrightarrow & 2q_0-3+\frac{2q_0+2}{2q_0^2+2q_0+1} \leq r_1+r_2 < 2q_0-2,
\end{eqnarray*}
a contradiction. In the case $j_1=2$ and $j_2=1$, we have $s_1=2$, $s_2=1$, $m_1\in \{0,1,2\}$, and $m_2=0$. Similarly, this yields
\begin{eqnarray*}
&& 2q_0(q-1)-q+1 \leq (r_1+r_2)\left(q+2q_0+1\right)+2q_0+3 < 2q_0(q-1) \\
&\Longrightarrow &  4q_0^3-2q_0^2-4q_0-2 \leq  (r_1+r_2)(2q_0^2+2q_0+1) < 4q_0^3-4q_0-3 \\
&\Longrightarrow & 2q_0-3+\frac{1}{2q_0^2+2q_0+1} \leq r_1+r_2 < 2q_0-2,
\end{eqnarray*}
a contradiction again.
\end{proof}

The following result shows that the pure gaps in Proposition~\ref{pg} are precisely the ones for which $n_1+n_2=2g-q$.

\begin{proposition}\label{PropSuzuki2}
Let $P_1,P_2 \in \mathcal{S}(\mathbb{F}_q)$ with $P_1\neq P_2$ and $q_0>2$. Then $(n_1,n_2)$ is a pure gap at $(P_1,P_2)$ with $n_1+n_2=2g-q$ if and only if
either $(n_1,n_2)$ or $(n_2,n_1)$ is equal to
$$
(\epsilon(q+2q_0+1)+1, 2g-q-1-\epsilon(q+2q_0+1))
$$
for some $\epsilon\in\{0,\dots, 2q_0-3\}$.
\end{proposition}

\begin{proof}
Proposition~\ref{pg} proves one direction. Conversely, we assume that $(n_1,n_2)$ is a pure gap at $(P_1,P_2)$ where $n_2=2g-q-n_1$. By \eqref{puregap}, we have that $(q-1) j_{1} < 2q-1$ and $(q-1)j_{2} < 2q-1$.
This means that $j_{1}, j_{2} \in \{1,2\}$.
\begin{enumerate}[{\normalfont(i)}]
\item When $j_{1}=1$, we have that $m_{1}=0$ and $s_{1}=1$. By (\ref{alpha}),  we get
$n_1=r_{1}(q+2q_0+1)+1.$ The claim follows with $\epsilon = r_{1}$.
\item When $j_{1}=s_{1}=2$, we have that $m_{1}\in\{0,1,2\}$. In this case
 \begin{eqnarray*}
n_2&=& 2q_0(q-1)-q-r_{1}(q+2q_0+1)-m_{1}q_0-2  \\
 &=&(2q_0-3-r_{1})(q+2q_0+1)+(2-m_{1})q_0+1.
 \end{eqnarray*}
We note that $2q_0-3-r_{1}\geq 0$. If $m_{1}=0$ then $m_{2}=2-m_{1}=2$ and $s_2=1$ imply  that  $j_{2}=q_0+1>3$.
 The choice  $m_{1}=1$ gives that $m_{2}=2-m_{1}=1$ and $s_2=1$. This means that  $j_{2}=q_0+1>3$, a contradiction. When $m_{1}=2$, we have that $n_2 = (2q_0-3-r_{1})(q+2q_0+1)+1$ \ has the desired form with $\epsilon = 2q_0-3-r_{1}$.
\end{enumerate}
The proof follows in a similar fashion by considering the same possibilities for $j_2$.
\end{proof}



\section{Applications of pure gaps to Coding Theory}

In this section we construct codes based on the families of pure gaps we obtained in Sections~\ref{Kummer} and~\ref{Suzuki}.
Given a curve $\mathcal{X}$, we fix  two distinct  rational places, $P_1$  and $P_2$. Let $(a_1,a_2)$ be a pure gap at $(P_1,P_2)$.
Following Theorem~\ref{distMany}, we consider  $(b_1,b_2):=(a_1,a_2)$ so that $G =(2a_1-1)P_1+(2a_2-1)P_2$ and $\deg(G)=2(a_1+a_2-1)$. We choose  $D$ to be the sum of all rational places on the curve $\mathcal{X}$ that are different from $P_1$ and $P_2$.
In all the cases we considered, we have that $2g-2 < \deg(G) < n=|\mathcal{X}(\mathbb{F}_q)|-2$, and therefore the parameters of the $C_\Omega(D,G)$ code satisfy $k = n+g-1-\deg(G)$ and $d\geq \deg (G)-2g+4$. Furthermore, our families of pure gaps  give rise   to codes with the same parameters, since $\deg(G)$ remains constant.

We denote by $R = k/n$ the relative minimum distance of the code $C_\Omega(D,G)$, and by $\delta = d/n$ the information rate.  By the Singleton bound, we have that $R+\delta \leq 1+1/n$. In Table \ref{table1} we present the parameters of our codes. With the exception
of the one based on  $\mathcal{X}_1$, they all satisfy $\displaystyle\lim_{n \rightarrow \infty} R+\delta =1$.

\begin{table}
\caption{Parameters of the codes}\label{table1}
\vspace*{0.3 cm}
\tabcolsep= 0.5 mm
{\footnotesize
\begin{tabular}{|c|c|c|c|c|c|c|}
\hline
$n$&$k$&$d\geq$&$\deg (G)$&Curve&Conditions&Ref.\\
\hline

\hline
$\begin{array}{c}q^8-q^6\\+q^5-1\\ \end{array}$&$\begin{array}{c}q^8-q^6-q^5/2+5q^3\\-7q^2/2+2q-2\\ \end{array}$&$\begin{array}{c}q^5-4q^3\\+3q^2-2q\\ \end{array}$&$\begin{array}{c}2q^5-6q^3\\+4q^2-2q-2\\ \end{array}$&GK&any $q$&Prop.~\ref{PropGK1}\\

\hline
$\begin{array}{c}(q^{2n}-q^n)m\\ +q^n-1\\ \end{array}$&$\begin{array}{c}(q^n-1)q^n(m-1)+\\(q^n-1)(m+3)/2-1\\ \end{array}$&$\begin{array}{c}(q^n-1)q^n-\\(q^n-1)m+4\\ \end{array}$&$(q^n-1)^2$&$\mathcal{X}_1$&$\begin{array}{c}m \mid (q^{2n}-1)\\ \gcd(m,q^n-1)=1\\  \end{array}$&Prop.~\ref{PropX_1}\\

\hline
$\begin{array}{c}q^4-q^2\\ +q-3\\ \end{array}$&$\begin{array}{c}q^4-3q^3+17q^2/2\\ -7q/2-6 \end{array}$&$\begin{array}{c}2q^3-9q^2\\+7q+4 \\ \end{array}$&$\begin{array}{c}4q^3-10q^2\\+2q+4\\ \end{array}$&$\mathcal{X}_2$&$q\equiv 0 \pmod 3$&Prop.~\ref{PropX_2}\\

\hline
$\begin{array}{c}(q^2-1)^2\\+2q-4\\ \end{array}$&$\begin{array}{c}q^4-3q^3+7q^2\\ -3q-4 \end{array}$&$\begin{array}{c}2q^3-8q^2+8q\\ \end{array}$&$\begin{array}{c}4q^3-10q^2\\+2q+4\\ \end{array}$&$\mathcal{X}_2$&$q\equiv 1 \pmod 3$&Prop.~\ref{PropX_2}\\

\hline
$q^4-3$&$\begin{array}{c}q^4-3q^3 +10q^2\\-4q-8 \end{array}$&$\begin{array}{c}2q^3-10q^2\\+6q+8\\ \end{array}$&$\begin{array}{c}4q^3-10q^2\\+2q+4\\ \end{array}$&$\mathcal{X}_2$&$q\equiv 2 \pmod 3$&Prop.~\ref{PropX_2}\\

\hline
$\begin{array}{c}q^{2n+2}-q^{2n}\\ -q^{n+3}\\+q^{n+2}-1\\ \end{array}$&$\begin{array}{c}q^{2n+2}-q^{2n}-q^{n+2}\\ -q^{n+2}/2+3q^n/2\\+4q^3-3q^2/2 \\\end{array}$&$\begin{array}{c}  q^{n+2}-q^n\\-3q^3+q^2+2  \\ \end{array}$&$\begin{array}{c}2q^{n+2}-2q^n\\ -4q^3+2q^2-2\\ \end{array}$&GGS&$\begin{array}{c}n\geq 5\\ n \textrm{ odd}\\ \end{array}$&Prop.~\ref{PropGGS1}\\

\hline
$q^2-1$&$\begin{array}{c}q^2-3q_0q\\+2q-3q_0-2\\ \end{array}$& $\begin{array}{c}2q_0q-2q\\-2q_0+4\\ \end{array}$ &$\begin{array}{c}4q_0q-2q\\-4q_0\\ \end{array}$&Suzuki&$\begin{array}{c}q_0=2^s\\ q=2q_0^2\\ s\geq 1\\ \end{array}$&Prop.~\ref{PropSuzuki2}\\

\hline
\end{tabular}
}
\end{table}


\section{Acknowledgements}
The first author was partially supported by the Italian Ministero dell'Istruzione, dell'Universit\`a
e della Ricerca (MIUR) and  the Gruppo Nazionale per le Strutture Algebriche, Geometriche e le loro Applicazioni (GNSAGA-INdAM).

The second author received support for this project provided by a PSC-CUNY award, \#60235-00 48, jointly funded by The Professional Staff Congress and The City University of New York.

Part of this work was carried out when the authors were visiting IMPA and UFRJ in Rio de Janeiro, Brazil.  We thank the warm hospitality we received from both institutes.


\end{document}